\documentclass{article}
\usepackage{header}

\usepackage[backend=bibtex,style=numeric,doi=false,isbn=false,url=false,maxnames=10]{biblatex}
\AtBeginDocument{
}

\AtEndDocument{
  \printbibliography
  \authaddresses
}

\begin{document}

\title{\sc Classifying affine structures with focus-focus singularities}
\author{Xiudi Tang}
\date{}
\maketitle

\begin{abstract}
  We study the singular affine structures of integrable systems with focus-focus fibers on the image of momentum maps.
  The classification of singular affine structures is equivalent to the classification of simple semitoric systems up to fiber-preserving symplectomorphisms but is not equivalent for semitoric systems with multiple pinched fibers and we give counterexamples for any number of more than one pinch on each fiber.
\end{abstract}

\section{Introduction}
\label{sec:intro}

The ``dimension reduction by half'' principle is central in the symplectic theory of finite-dimensional integrable systems, initiated from that diffeomorphic manifolds have symplectomorphic cotangent bundles.
Recall that an \emph{integrable system} on a symplectic manifold $(M, \om)$ of dimension $2n$ is a smooth almost everywhere submersion $F \colon M \to \R^n$ with Poisson commutative components.
If every component of $F$ has a $2\pi$-periodic Hamiltonian flow---called \emph{toric}---there is a celebrated theorem by Atiyah~\cite{MR642416}, Guillemin--Sternberg~\cite{MR664117}, and Delzant~\cite{MR984900}, saying that the isomorphism classes of compact toric systems are bijective to Delzant polytopes $\De \in \R^n$ up to tropical affine transformations; in this case, the integrable system is determined by the image of the momentum map, and this observation yields tons of results and studies.
The $\R^n$ is naturally seen as a tropical affine manifold, while $\De$ is a tropical affine submanifold with corners.
In general, an integrable system may have various singular fibers and $B = F(M)$ has singularity and monodromy which obstruct its affine embedding into $\R^n$.
Then $F \colon (M, \om) \to B$ is nothing but a \emph{singular Lagrangian fibration}.
We quote the folklore principle from~\cite{MR3868426}.
\begin{quotation}
  Let  $F \colon (M, \om) \to B$ and $F' \colon (M', \om') \to B'$ be two singular Lagrangian fibrations on symplectic manifolds.
  If $B$ and $B'$ are affinely equivalent (as stratified manifolds with singular affine structures), then these Lagrangian fibrations are fiberwise symplectomorphic.
\end{quotation}

This principle is true as is in many but \textbf{not all} cases.
It is true in the toric and the two-dimensional cases, as Dufour--Molino--Toulet~\cite{MR1278158} showed that integrable systems of dimension two are classified by their affine Reeb graphs.
It is crucial to understand the singularity and monodromy of the affine structure of $B$.
This is known for semitoric and almost-toric systems of dimension four by~\cite{MR1941440,PeTa2018}.
More than these, a complete classification of simple semitoric systems up to fiber-preserving symplectomorphisms is due to~\cite{2019arXiv190903501v3P} and that of almost-toric systems up to symplectomorphisms is seen in~\cite{MR2670163}.

In this paper, we aim to answer~\cite[Problem 2.2]{MR3868426} on the principle above, for semitoric and almost-toric systems of dimension four.
Of course, the answers depend on the actual definition of the ``stratified manifolds with singular affine structures''.
We make use of the \emph{period lattice} over both regular and singular values of $F$.
In \cref{sec:int-sys} a preliminary on integrable systems and their invariants is provided.
In \cref{sec:affine-structure} there states the precise meaning of the affine structures and their isomorphisms, so as to set a clear goal of the problem.
In \cref{sec:main-result} we give our answers to this question.
In \cref{ssec:main-theorem-single-pinch} we prove that this principle always holds for simple semitoric systems.
In \cref{ssec:first-order} we quote the first-order smooth invariants from Bolsinov--Izosimov~\cite{MR4057723}.
In \cref{ssec:main-theorem-multiple-pinch} we show that this principle does not hold for semitoric systems in general by trivial counterexamples for twice pinched fibers and nontrivial counterexamples for three times pinched fibers in \cref{ssec:counterexamples}.

\subsection*{Acknowledgements}

The author thanks the organizers of Finite Dimensional Integrable Systems in Geometry and Mathematical Physics in 2017 at the Centre de Recerca Matem\`atica where many interesting problems were collected, organized, and announced.
This work was initiated in 2018 at Cornell University.
The author is grateful for helpful discussions with Reyer Sjamaar, Qingsheng Zhang, and Xuhui Zhang.
The author thanks the Beijing Institute of Technology for a Research Fund Program
for Young Scholars and National Natural Science Foundation of China for a Young Scientist Fund.

\section{Integrable systems with focus-focus fibers}
\label{sec:int-sys}

\begin{definition} \label{def:ator-stor}
  An integrable system $F$ on $(M, \om)$ is \emph{almost-toric} if all singular points of $F$ are nondegenerate and without hyperbolic components.
  Moreover, we assume $F$ to be proper and have connected fibers.

  An integrable system $F = (J, H)$ on $(M, \om)$ is \emph{semitoric} if it is almost-toric, and $J$ is proper and has $2\pi$-periodic Hamiltonian flow.
\end{definition}

The notions in \cref{def:ator-stor} find their origins in~\cite{MR1852084,MR2024634,MR2304341}.

\begin{definition} \label{def:isom-int-sys}
  A symplectomorphism $\varphi \colon M \to M'$ is an \emph{isomorphism} from an integrable system $F$ on $(M, \om)$ to $F'$ on $(M', \om')$ if it lifts a diffeomorphism $G \colon F(M) \to F'(M')$.
\end{definition}

Almost-toric systems adapt to \cref{def:isom-int-sys} but semitoric ones have a more dedicated definition of isomorphisms making use of the global $\bbS^1$-action.

\begin{definition} \label{def:isom-int-sys-stor}
  A symplectomorphism $\varphi \colon M \to M'$ is a \emph{semitoric isomorphism} from an semitoric system $F$ on $(M, \om)$ to $F'$ on $(M', \om')$ of dimension four if it lifts a diffeomorphism $G \colon F(M) \to F'(M')$ in the form of $G(c^1, c^2) = (c^1, g(c^1, c^2))$ for some smooth function $g$ such that $\frac{\partial g}{\partial c^2}>0$ everywhere.
\end{definition}

Similarly, one could define local and semi-local isomorphisms, respectively, on the germs of neighborhoods of a point and germs of saturated neighborhoods of a fiber, respectively, and say local and semi-local classifications of these germs.
By Eliasson~\cite{Eliasson1984}, the local equivalence classes are the same as algebraic/topological classes, which are cartesian products of regular, elliptic, hyperbolic, and focus-focus normal forms.
In semitoric and almost-toric systems, there is no hyperbolic component, and when the dimension is four the only local types of singular points are elliptic-regular, elliptic-elliptic, which are the same as toric cases, and focus-focus, which is new.
By Zung~\cite{MR1389366}, a focus-focus point is isolated and has a saturated neighborhood with only finitely many singular points, all of the focus-focus type on the same fiber, which we call a \emph{focus-focus fiber}.
Only at a focus-focus fiber do the semi-local germs have more than one equivalence class, classified by Pelayo--Tang~\cite{PeTa2018}.

\begin{definition} \label{def:diff-positive}
  Let $\frakV$ be the set of germs of diffeomorphisms $G$ of $\R^2$ fixing $0$ at $0$ with $G(c^1, c^2) = (c^1, g(c^1, c^2))$ for some smooth function $g$, and let $\frakV_+$ be the subset of $\frakV$ of those with $\frac{\partial g}{\partial c^2} > 0$.
\end{definition}

\begin{definition} [{\cite[(3.3)]{PeTa2018}}] \label{def:ff-label}
  Let $\sfR$ be the space of formal power series in two variables $X, Y$, without the constant term, $\sfR_{2\pi X} \defeq \sfR / (2\pi X)\Z$, and let $\sfR_+$ be the group of formal power series in $\sfR$ with positive $Y$-coefficients.
  A \emph{focus-focus label} of \emph{multiplicity} $m \in \N$ is $\sfl = [\sfs_j, \sfg_{j, \ell}]_{j, \ell \in \Z_m} \in \sfR_{2\pi X}^m \times \sfR_+^{m^2}$ satisfying the following relations:
  \begin{equation*}
    \begin{cases}
      \sfs_j(X, Y) = \sfs_\ell(X, \sfg_{j, \ell}(X, Y)), \\
      \sfg_{j, j}(X, Y) = Y, \\
      \sfg_{j, p}(X, Y) = \sfg_{\ell, p}(X, \sfg_{j, \ell}(X, Y)),
    \end{cases}
  \end{equation*}
  for $j, \ell, p \in \Z_m$ under the action by $\Z_m$ on the indices.
\end{definition}

\begin{theorem} [{\cite[Theorem~1.1]{PeTa2018}}] \label{thm:invariant-bijection-stor}
  An orbit of semi-local germs at a focus-focus fiber under the action of $\frakV_+$ is one-to-one correspondent to a focus-focus label of the same multiplicity.
\end{theorem}

Under the composition $\frakV_+ \subset \frakV$ is a normal subgroup with quotient group $\Z_2$ generated by the reflection about the abscissa axis.

\begin{theorem} [{\cite[Theorem~1.1]{PeTa2018}}] \label{thm:invariant-bijection-ator}
  An orbit of semi-local germs at a focus-focus fiber of multiplicity $m$ under the action of $\frakV$ is one-to-one correspondent to the orbit of a focus-focus label under the action of $\Z_2$ generated by $[\sfs_j, \sfg_{j, \ell}]_{j, \ell \in \Z_m} \mapsto [\sfs'_j, \sfg'_{j, \ell}]_{j, \ell \in \Z_m}$ where
  \begin{align*}
    \sfs'_j(X, Y) &= \sfs_j(-X, Y) + k \pi X, &
    \sfg'_{j, \ell}(X, Y) &= \sfg_{j, \ell}(-X, Y).
  \end{align*}
\end{theorem}

Based on the semi-local classification at focus-focus fibers, a global classification of isomorphism classes of semitoric systems was proven.

\begin{definition} \label{def:complete-ff-label}
  Let $\tsfR = \R[[X, Y]]$ be the space of formal power series in two variables $X, Y$.
  A \emph{complete focus-focus label} of \emph{multiplicity} $m \in \N$ is $\tsfl = [\tsfs_j, \sfg_{j, \ell}]_{j, \ell \in \Z_m} \in \tsfR^m \times \sfR_+^{m^2}$ satisfying the following relations:
  \begin{equation*}
    \begin{cases}
      \tsfs_j(X, Y) = \tsfs_\ell(X, \sfg_{j, \ell}(X, Y)), \\
      \sfg_{j, j}(X, Y) = Y, \\
      \sfg_{j, p}(X, Y) = \sfg_{\ell, p}(X, \sfg_{j, \ell}(X, Y)),
    \end{cases}
  \end{equation*}
  for $j, \ell, p \in \Z_m$ under the action by $\Z_m$ on the indices.
\end{definition}

\begin{definition} \label{def:stor-polygon-corner}
  Let
  \begin{equation*}
    T = \begin{pmatrix} 1 & 0 \\ 1 & 1 \end{pmatrix}.
  \end{equation*}
  Let $v$ be a vertex of a convex polygon, and let $\xi_1, \xi_2 \in \Z^2$ be the primitive vectors directing the edges emanating from $v$ in the counter-clockwise order and let $s \in \N$.
  Then $v$ is
  \begin{itemize}[noitemsep]
    \item a \emph{Delzant corner} if $\det(\xi_1, \xi_2) = \pm 1$;
    \item an \emph{$s$-fake corner} if $\det(\xi_1, T^s \xi_2)=0$; or
    \item an \emph{$s$-hidden corner} if $\det(\xi_1, T^s \xi_2) = \pm 1$.
  \end{itemize}
\end{definition}

\begin{definition} [{\cite[Definition 4.4]{2019arXiv190903501v3P}}]\label{def:markedpoly-rep}
  A \emph{complete semitoric ingredient representative} is a tuple
  \begin{equation*}
    \Pa{\De, (c_i)_{i = 1}^\vf, \Pa{[\tsfs^i_j, \sfg^i_{j, \ell}]_{j, \ell \in \Z_{m_i}}}_{i = 1}^\vf}
  \end{equation*}
  where $\De$ is a convex polygon with $\proj_1 \colon \De \to \R$ proper, $(c_i)_{i = 1}^\vf$ is a tuple of points in lexicographic order, and for $i \in \Seq{1, \ldots, \vf}$ the point $c_i = (c_i^1, c_i^2)$ lies in the interior of $\De$ and $[\tsfs^i_j, \sfg^i_{j, \ell}]_{j, \ell \in \Z_{m_i}}$ is complete focus-focus label of multiplicity $m_i$ such that $(\tsfs^i_j)^{(0, 0)} = 2\pi c_i^2$ for all $j \in \Z_{m_i}$ and any vertex $v$ of $\De$ is either an $s_i$-hidden or $s_i$-fake corner if $v \in \ell_{c_i}^+$ for some $i \in \Seq{1, \ldots, \vf}$ where $\ell_{c_i}^+ = \Set{c_i^1} \times [c_i^2, \infty)$ and $s_i = \sum_{1 \leq k \leq \vf, c_k^1 = c_i^1} m_k$, while $v$ is a Delzant corner otherwise.
\end{definition}

\begin{theorem} [{\cite[Theorem 4.11]{2019arXiv190903501v3P}}] \label{thm:markedpoly-classification}
  The isomorphism class of a semitoric system of dimension four is one-to-one correspondent to a complete semitoric ingredient representative under the action of $\Z \times \R$ by for $(k, b)$ applying a shear transformation by $T^k$ and a translation by $(0, b)$ on the polygon and points, and on each complete focus-focus label by
  \begin{align*}
    [\tsfs^i_j, \sfg^i_{j, \ell}]_{j, \ell \in \Z_{m_i}} \mapsto [\tsfs^i_j + 2\pi (k X + b), \sfg^i_{j, \ell}]_{j, \ell \in \Z_{m_i}}.
  \end{align*}
\end{theorem}

\section{Tropical affine manifolds with singularities}
\label{sec:affine-structure}

\subsection{Tropical affine structures}
\label{ssec:affine-structure}

Let $F \colon M \to \R^n$ be an integrable system on $(M, \om)$ with compact and connected fibers and let $B = F(M)$.
In this section, we describe the affine structures of $B$ induced by a moment map $F$.
We use $T^* B$ to denote $\Res{T^* \R^2}_B$.
Let $\calN(B, b)$ denote the collection of open neighborhoods of $b \in B$ in $B$.

For any open subset $U \subset B$, and each $\beta \in \Om^1(U)$, let $X_{\beta} = -\om^{-1} (F^* \beta) \in \frakX(M)$.
Let $\Psi_\beta \in \Diff(F^{-1}(U))$ be the time-$1$ map of $X_\beta$.
Then $\Psi$ gives the action of $\Gamma(T^*U)$ on $F^{-1}(U)$ by fiber-preserving diffeomorphisms.
For any point $b \in B$, and each $\beta_b \in T^*_b B$, let $X_{\beta_b} = -\om^{-1} (F^* \beta_b) \in \frakX(F^{-1}(b))$.
Let $\Psi_{\beta_b} \in \Diff(F^{-1}(b))$ be the time-$1$ map of $X_{\beta_b}$.
Then $\Psi$ gives the action of $T^*_b B$ on $F^{-1}(b)$ by diffeomorphisms.
\footnote{A singular fiber $F^{-1}(b)$ may not be a submanifold of $M$; however, each $\beta_b \in T^*_b B$ can be extended to a smooth $1$-form near $b$ so the smoothness of $X_{\beta}$ and $\Psi_{\beta_b}$ makes sense.}

\begin{definition} \label{def:affine-structures}
  Let
  \begin{align*}
    \check L^F_b &= \Set{\alpha_b \in T^*_b B \mmid \Psi_{2\pi \alpha_b} = \identity_{F^{-1}(b)}} \\
    \Period^F(U) &= \Set{\alpha \in \Om^1(U) \mmid \Psi_{2\pi \alpha} = \identity_{F^{-1}(U)}} \\
    \calA^F(U) &= \Set{A \in \rm\rmC^\infty(U) \mmid \Psi_{2\pi \der A} = \identity_{F^{-1}(U)}}
  \end{align*}
  for any $b \in B$ and any open $U \subseteq B$.
  Let $\check L^F = \bigcup_{b \in B} \check L^F_b \subseteq T^* B$ with a smooth surjection $\check \pi^F \colon \check L^F \to B$ induced from the natural projection $\rmT^* B \to B$.

  We call $\check L^F$ the \emph{period lattice} of $F$ and $\check \pi^F$ the \emph{period lattice bundle} of $F$.
  Then $\Period^F$ is a sheaf of abelian groups over $B$, called the \emph{period sheaf} of $F$; the local sections of $\Period^F$ are \emph{period forms}.
  We call $\calA^F$ the \emph{action sheaf} of $F$ or the \emph{singular affine structure} on $B$, with the local sections \emph{local action coordinates}.
\end{definition}

By the action-angle theorem, on any simply connected open set $U$ of regular values of $F$, there are action integrals $A_1, \dotsc, A_n \colon U \to \R$, in which case, $(\alpha_1, \dotsc, \alpha_n) = (\der A_1, \dotsc, \der A_n)$ is a $\Z$-basis of $\Period^F(U)$.
We ask when $B$ carries a natural structure of tropical affine manifold with singularity.

\begin{definition} [{\cite[Definition 1.22 and 1.24]{MR2722115}}]
  A \emph{tropical affine manifold} is a topological manifold with an atlas of coordinate charts $A \colon U \to \R^n$ with transition functions $A \circ (A')^{-1} \in \rmGL(n, \Z) \ltimes \R^n \defeq \Trop(\R^n)$.\footnote{The notation $\Trop(\R^n)$ is for the tropical affine group; some authors call it the integral affine group which sometimes mean $\rmGL(n, \Z) \ltimes \Z^n$, though.}

  A \emph{tropical affine manifold with singularities} is a topological manifold with a subset called the \emph{singular locus}, which is a locally finite union of locally closed submanifolds of codimension at least two, and a tropical affine structure on the complement of the singular locus.

  A \emph{tropical affine manifold with corners and singularities} is a topological manifold along with the singular locus and a tropical affine structure with corners on the complement of the singular locus.
\end{definition}

\begin{lemma}
  If $F$ is a semitoric or almost-toric system on $(M, \om)$, then $B = F(M)$ is naturally a tropical affine manifold with the set of nodes as the singular locus and the action integrals as coordinate charts.
\end{lemma}

\begin{proof}
  The classification of the points in $M$ \emph{\`a la} Eliasson sorts the values of $F$: regular values with all regular preimages, boundary values with all elliptic-regular preimages, corner values with a unique elliptic-elliptic preimage for each, and nodes with regular preimages and finitely many focus-focus preimages.
  The regular, boundary, and corner values are analogous to those of toric momentum maps and constitute the strata of $B$ of dimensions two, one, and zero, respectively.
  The nodes are isolated and therefore constitute the singular locus of $B$ of codimension two.

  Let $U$ be a simply connected open set of regular values of $F$.
  Suppose we have two different action integrals $A = (A_1, A_2), A' = (A'_1, A'_2) \in \calA^F(U)$.
  Then since both $(\der A_1, \der A_2)$ and $(\der A'_1, \der A'_2)$ form bases of $\Period^F(U)$, there are $P \in \rmGL(n, \Z)$ and $T \in \R^n$ such that $A = P \circ A' + T$.
\end{proof}

\begin{proposition}
  Let $F \colon M \to \R^n$ be an integrable system on $(M, \om)$ with compact and connected fibers and let $B = F(M)$.
  Then the period lattice $\check L^F$ determines the period sheaf $\Period^F$ and the action sheaf $\calA^F$, and the latter two determine each other.
\end{proposition}

\begin{proof}
  By \cref{def:affine-structures}, given an open subset $U$ of $B$, $\beta \in \Period^F(U)$ iff $\beta_b \in \check L^F_b$ for any $b \in U$, and $A \in \calA^F(U)$ iff $\der A \in \Period^F(U)$.
  Then $\check L^F$ determines $\Period^F$, and $\Period^F$ determines $\calA^F$.
  Let $U$ be a shrinkable open subset of $B$.
  Since $\Period^F(U) \subseteq \rmZ^1(U) = \rmB^1(U)$, then $\beta \in \Period^F(U)$ iff there is $A \in \calA^F(U)$, uniquely up to adding a real number, such that $\der A = \beta$.
  Then $\calA^F$ determines $\Period^F$ as shrinkable open subsets form a topological base of $B$.
\end{proof}

\subsection{Local symplectic and affine normal forms}
\label{ssec:local-normal-form}

By Eliasson's linearization theorem for non-degenerate focus-focus points~\cite{Eliasson1984, MR3098203}, there is no local invariant near a singular point.

\begin{definition} \label{def:eliasson-focus}
  An \emph{Eliasson local chart} at a focus-focus point $p \in M$ of $F$ is an isomorphism $\psi$ from the integrable system $\Res{F}_V$ on $(V, \om)$ to $\Res{q}_{V_0}$ on $(V_0, \omc)$ where $V \in \calN(M, p)$, $V_0 \in \calN(\R^4, 0)$, and $q$ is given by
  \begin{equation} \label{def:local-model-focus}
    q(x_1, y_1, x_2, y_2) = (x_1 y_2 - x_2 y_1, x_1 y_1 + x_2 y_2).
  \end{equation}
  In other words, the following diagram commutes
  \begin{equation*}
    \xymatrix{
      (V, \om) \ar[r]^{\psi} \ar[d]^{F} & (V_0, \omc) \ar[d]^{q} \\
      U \ar[r]^{E} & U_0
    }.
  \end{equation*}
  where $E$ is a diffeomorphism from $U \in \calN(B, 0)$ to $U_0 \in \calN(\R^2, 0)$.
\end{definition}

\begin{theorem} \label{thm:eliasson-focus}
  Let $p \in M$ be a focus-focus point of an integrable system $F$ on $(M, \om)$.
  \begin{itemize}
    \item There exists an Eliasson local chart at $p$.
    \item If $\psi$ and $\psi'$ are two Eliasson local charts at $p$ lifting $E$ and $E'$ respectively, then $G = E' \circ E^{-1}$ has the form $G(c^1, c^2) = (\epsilon_1 c^1, \epsilon_2 c^2 + \inftes)$, with $\epsilon_i = \pm 1$, $i = 1, 2$.
    \item If $\psi$ is an Eliasson local chart at $p$ lifting $E$ and $G(c^1, c^2) = (\epsilon_1 c^1, \epsilon_2 c^2 + \inftes)$, with $\epsilon_i = \pm 1$, $i = 1, 2$, then there is an Eliasson local chart $\psi'$ at $p$ lifting $E' = G \circ E^{-1}$.
  \end{itemize}
\end{theorem}

\begin{proof}
  See~\cite[Lemma 4.1, Lemma 5.1]{MR1941440}.
\end{proof}

\begin{corollary} \label{cor:eliasson-focus-stor}
  Let $p \in M$ be a focus-focus point of a semitoric system $F$ on $(M, \om)$.
  \begin{itemize}
    \item There exists an Eliasson local chart $\psi$ at $p$ lifting $E$ such that $E$ preserves the orientation and acts on the abscissa by translation
    \begin{equation*}
      \proj_1 \circ E = \proj_1 - \proj_1(p).
    \end{equation*}
    \item If $\psi$ and $\psi'$ are two Eliasson local charts at $p$ lifting $E$ and $E'$respectively, such that $E$ and $E'$ preserve the orientation and act on the abscissa by translation then $G = E' \circ E^{-1}$ has the form $G(c^1, c^2) = (c^1, c^2 + \inftes)$.
  \end{itemize}
\end{corollary}

\begin{proof}
  This is a corollary of \cref{thm:eliasson-focus,def:isom-int-sys-stor}.
\end{proof}

\begin{definition} \label{def:inf-jet-positive}
  Let $\Tl_0 \colon \Diff_0(\R^2) \to \C[[Z, \cj{Z}]]$ be taking the infinite jet at the origin, then $\Tl_0$ is a surjective group homomorphism.
  Let $\sfV = \Tl_0(\frakV)$ and $\sfV_+ = \Tl_0(\frakV_+)$.
\end{definition}

Let $F = (J, H)$ be a semitoric system on $(M, \om)$ and let $B = F(M)$.
Let $M_\rmf \subset M$ be the set of \emph{focus-focus points} of $F$ and let $B_\rmf = F(M_\rmf)$ be called the set of \emph{nodes} of the \emph{base} $(B, \Period^F)$.
Arrange $B_\rmf$ lexicographically into $(\tilde{c}_i)_{i = 1}^\vf$ and let the fiber at $\tilde{c}_i$ has multiplicity $m_i \in \N$.
Let $B_+ = B \setminus \bigcup_{i = 1}^\vf \ell_{\tilde{c}_i}^+$.
By~\cite[Proposition 2.5]{2019arXiv190903501v3P} there exists an injective, orientation preserving, continuous function $A_+ \colon B \to \R^2$, called a choice of \emph{piecewise affine coordinates} as in~\cite[Definition 2.6]{2019arXiv190903501v3P} with the following properties:
\begin{enumerate}
  \item $A_+$ preserves the abscissa;
  \item the restriction $A_+|_{B_+}$ of $A_+$ to $B_+$ is smooth;
  \item $\Res{A_+}_{B_+}^* \check L^p = \check L^F$, where $\check L^p = \R^2 \times \Z^2$ is the standard lattice on $\R^2$ given by the projection $p \colon (\R^2 \times T^* (\R/2\pi\Z)^2, \omc) \to \R^2$.
\end{enumerate}
Let $\ln_+ \colon \C \setminus \imag \R^+ \to \C$ be the determination of $\ln$ with $\ln_+ 1 = 0$ and branch cut at $\imag \R^+$.
Let $K_+ \colon \C \setminus \imag \R^+ \to \R$ be given by
\begin{equation*}
  K_+ (c) = -\Im (c \ln_+ c - c).
\end{equation*}

The correspondence in \cref{thm:markedpoly-classification} is the map $\tilde{\rmi}$ in \cite[(3.1)]{2019arXiv190903501v3P} with all the lower wall-crossing indices equal to zero.
More concretely, we have $\De = A_+(B)$ and $(c_i)_{i = 1}^\vf = (A_+(\tilde{c}_i))_{i = 1}^\vf$.
Consider a node $\tilde{c}_i \in B_\rmf$ and let $(p^i_j)_{j \in \Z_{m_i}}$ be the tuple of focus-focus points in $F^{-1}(\tilde{c}_i)$ in order according to the direction of the flow of $X_H$.
At any $p^i_j$ let $E^i_j \colon U_i \to \R^2 \cong \C$ be an Eliasson local chart satisfying the first item of \cref{cor:eliasson-focus-stor}.
Let $\wt S^i \colon U_i \to \R$ be the smooth extension of
\begin{equation} \label{eq:def-S-tilde}
  2\pi \proj_2 \circ A_+ - \sum_{j \in \Z_{m_i}} (E^i_j)^* K_+ + 2\pi \sum_{1 \leq k < i, c_k^1 = c_i^1} m_k (c_i^1 - \proj_2)\Heaviside_{(c_i^1 - \proj_1)},
\end{equation}
where $\Heaviside_x$ is $1$ when $x \geq 0$ and $0$ for $x < 0$, which is the Heaviside function.
The complete focus-focus label $\tsfl^i = [\tsfs^i_j, \sfg^i_{j, \ell}]_{j, \ell \in \Z_{m_i}}$ is given by
\begin{equation*} \begin{aligned}
  \tsfs^i_j(X, Y) = \Tl_0[\wt S^i\circ (E^i_j)^{-1}], \qquad \sfg^i_{j, \ell}(X, Y) = \Tl_0[\proj_2 \circ E^i_\ell \circ (E^i_j)^{-1}],
\end{aligned} \end{equation*}
for $j, \ell \in \Z_{m_i}$.

By the action-angle theorem, the $i$-stratum $B^{(i)}$ of $B \setminus B_\rmf$, $i \in \Set{0, 1, 2}$, is the set of values over which the points have Eliasson type $\bfk = (i, n-i, 0, 0)$.
Then there is a (set-theoretic) decomposition
\begin{equation*}
  B = B^{(2)} \sqcup B^{(1)} \sqcup B^{(0)} \sqcup B_\rmf.
\end{equation*}

Let $\kappa$ denote the multi-valued $1$-form on $\R^2 \setminus \Set{0}$ given by
\begin{equation*}
  \kappa = -\Im \ln c \der c^1 - \Re \ln c \der c^2
\end{equation*}
with monodromy equal to $-2\pi \der c^1$, as in \cite[(2.4)]{PeTa2018}.

\begin{theorem} \label{thm:affine-local-form}
  Let $F = (J, H)$ be a semitoric system on $(M, \om)$ of dimension four and let $B = F(M)$.
  Then for any $i \in \Seq{1, \ldots, \vf}$ there is $\wt S^i \colon U_i \to \R$ as in \cref{eq:def-S-tilde} and the period lattice $\check L^F$ has local normal forms at $b \in B$ as
  \begin{equation*}
    \begin{cases}
      \check L^F_b \simeq \Z^2, & b \in B^{(2)} \\
      \check L^F_b \simeq \R \times \Z, & b \in B^{(1)} \\
      \check L^F_b = T^*_b B \simeq \R^2, & b \in B^{(0)}; \\
      \check L^F_b = (\der c^1)_b \Z \oplus \frac{1}{2\pi}((\der \wt S^i)_b + \sum_{j \in \Z_{m_i}} (E^i_j)^* \kappa_b) \Z & b \in U_i \setminus \Set{\tilde{c}_i} \\
      \check L^F_b = (\der c^1)_b \Z, & b \in B_\rmf
    \end{cases}
  \end{equation*}
  where $\kappa_b$ is any representative of $\kappa$ at $b$.
  In particular, the period lattice $\check L^F$, the period sheaf $\Period^F$, and the action sheaf $\calA^F$ determine one another.
\end{theorem}

\begin{proof}
  The local normal forms of $\check L^F$ are a result of \cite{2019arXiv190903501v3P}.
  Since for every $\beta_b \in \check L^F_b$ there is a $U \in \calN(B, b)$ and a $\beta \in \Gamma(U; \Res{\check L^F}_U)$ with value $\beta_b$ at $b$, and that
  \begin{equation*}
    \Period^F(U) = \Gamma(U; \textstyle\Res{\check L^F}_U),
  \end{equation*}
  we obtain that $(\check L^F, \check \pi^F)$ is the \'etal\'e space of $\Period^F$, which is, in particular, determined by $\Period^F$.
\end{proof}

\subsection{Smooth isomorphisms of the singular affine structure}
\label{ssec:isomorphism-of-affine-structure}

\begin{theorem} \label{thm:lattice-bundle-period-action-stor}
  Let $F$ be a semitoric system on $(M, \om)$ and $F'$ on $(M', \om')$ of dimension four.
  Then the following are equivalent: $G^* \check L^{F'} = \check L^F$, $G^* \Period^{F'} = \Period^F$, and $G^* \calA^{F'} = \calA^F$.
\end{theorem}

\begin{proof}
  This is a corollary of \cref{thm:affine-local-form}.
\end{proof}

\begin{corollary}
  Let $F$ be a semitoric system on $(M, \om)$ and $F'$ on $(M', \om')$ and let $B = F(M)$ and $B' = F'(M')$ of dimension four.
  Then there is a diffeomorphism $G \colon B \to B'$ with $G^* \check L^{F'} = \check L^F$ iff there are complete semitoric ingredient representatives
  \begin{equation*}
    \Pa{\De, (c_i)_{i = 1}^\vf, \Pa{[\tsfs^i_j, \sfg^i_{j, \ell}]_{j, \ell \in \Z_{m_i}}}_{i = 1}^\vf}, \Pa{\De', (c'_i)_{i = 1}^\vf, \Pa{[\tsfs^{\prime i}_j, \sfg^{\prime i}_{j, \ell}]_{j, \ell \in \Z_{m'_i}}}_{i = 1}^\vf}
  \end{equation*}
  where $\De = \De'$, $(c_i)_{i = 1}^\vf = (c'_i)_{i = 1}^\vf$, and with $(\tilde{c}_i)_{i = 1}^\vf$ the lexicographic ordering of $B_\rmf$ and $(\tilde{c}'_i)_{i = 1}^\vf$ of $B'_\rmf$, there is $U_i \in \calN(B, \tilde{c}_i)$, $U'_i \in \calN(B', \tilde{c}'_i)$, and a diffeomorphism $G_i \colon U_i \to U'_i$ such that $G_i^* \Res{\check L^{F'}}_{U'_i} = \Res{\check L^F}_{U_i}$ for any $i \in \Seq{1, \ldots, \vf}$.
\end{corollary}

\begin{proof}
  Let $B_+ = B \setminus \bigcup_{i = 1}^\vf \ell_{\tilde{c}_i}^+$ and $B'_+ = B' \setminus \bigcup_{i = 1}^\vf \ell_{\tilde{c}'_i}^+$.
  By~\cite[Proposition 2.5]{2019arXiv190903501v3P} there exist choices of piecewise affine coordinates $A_+ \colon B \to \R^2$ and $A'_+ \colon B' \to \R^2$.
  Suppose such a $G$ exists and by the 3\textsuperscript{rd} of their properties the map
  \begin{equation*}
    \Res{A'_+}_{B'_+} \circ G \circ \Res{A_+}_{B_+}^{-1} \colon \De \setminus \bigcup_{i = 1}^\vf \ell_{c_i}^+ \to \De' \setminus \bigcup_{i = 1}^\vf \ell_{c'_i}^+
  \end{equation*}
  belongs to the action of $\Z \times \R$ described in \cref{thm:markedpoly-classification} and therefore $A'_+ \circ G \circ A_+^{-1}$ equals the identity for appropriate choice of $A_+$ and $A'_+$.
  Then we have $\De = A_+(B) = A'_+(B') = \De'$, $(c_i)_{i = 1}^\vf = (A_+(\tilde{c}_i))_{i = 1}^\vf = (A'_+(\tilde{c}'_i))_{i = 1}^\vf = (c'_i)_{i = 1}^\vf$, and since $G$ maps $\tilde{c}_i$ to $\tilde{c}'_i$ then $G_i^* \Res{\check L^{F'}}_{U'_i} = \Res{\check L^F}_{U_i}$ is automatic with $G_i = \Res{G}_{U_i}$.

  Suppose that $G_i \colon U_i \to U'_i$ are given with $G_i^* \Res{\check L^{F'}}_{U'_i} = \Res{\check L^F}_{U_i}$ and that $\De = \De'$, $(c_i)_{i = 1}^\vf = (c'_i)_{i = 1}^\vf$.
  Let $A_+ \colon B \to \R^2$ and $A'_+ \colon B' \to \R^2$ be the corresponding choices of piecewise affine coordinates.
  Since $\Res{A_+}_{B_+}$ and $\Res{A'_+}_{B'_+}$ preserve the period lattices then for $G_+ = (A'_+)^{-1} \circ A_+$ we have $G_+^* \Res{\check L^{F'}}_{B'_+} = \Res{\check L^F}_{B_+}$.
  In the simply connected set $U_i \cap B_+$, both $G_+$ and $G_i$ preserve the period lattices, and so they differ by the action of $\Z \times \R$ described in \cref{thm:markedpoly-classification}.
  By continuity $G_+^* \Res{\check L^{F'}}_{U'_i} = \Res{\check L^F}_{U_i}$.
  Now $G_+$ preserves the period lattices on $B_+ \cup \bigcup_{i = 1}^\vf U_i$ whose complement in $B$ has no nodes, and therefore $G_+$ can be extended to a diffeomorphism $G \colon B \to B'$ with $G^* \check L^{F'} = \check L^F$.
\end{proof}

For local normal forms of $\calA^F$, we obtain that
\begin{equation} \label{eq:action-local-form}
  \calA^F(U) = \Res{A_1}_U \Z \oplus \Res{A_2}_U \Z \oplus \R
\end{equation}
for any simply connected open subset $U \subseteq B^{(2)}$, and when $U \subseteq U_i \setminus \Set{\tilde{c}_i}$ for some $i \in \Seq{1, \ldots, \vf}$ then we can choose
\begin{equation} \begin{cases} \label{eq:action-local-form-basis}
  A_1 = c^1, \\
  A_2 = \frac{1}{2\pi} \Pa{\wt S^i - \sum_{j \in \Z_{m_i}} \Im(E^i_j \ln E^i_j - E^i_j)}.
\end{cases} \end{equation}
But when we consider a region across $\tilde{c}_i$ the $\ln$ function becomes multi-valued and we study their function spaces in the next section.

\subsection{Multi-valued functions and infinite jets}
\label{sec:infinite-jet}

Note that $\frakV_+$ consists of the diffeomorphisms lifted by semitoric isomorphisms.
But on the right-hand side, the function is not the usual meaning of function.
To clarify we would give an appropriate definition of multi-valued functions they live in.
For an $m$-pinched singular fiber we consider and classify ``functions'' of the form of $\sum_{j \in \Z_m} G_j \ln G_j$ for $G_j \in \frakV_+$; we make this precise.

\begin{definition} \label{def:multi-valued-function}
  Let $\Sigma = \C \setminus \Seq{0}$ denote the punctured complex plane, and let $\wt\Sigma$ be the universal cover of $\Sigma$ which is the Riemann surface for $\ln$ and $\pi \colon \wt\Sigma \to \Sigma$ be the covering map.
  Let $\calS$ be the space of germs of smooth functions $f \colon \wt\Sigma \to \C$ whose values on each fiber of $\pi$ form an arithmetic sequence and $f \sim g$ iff $f = g$ on $\pi^{-1}(U \cap \Sigma)$ for some neighborhood $U$ of $0$ in $\C$.
  An element of $\calS$ is called \emph{germs of multi-valued smooth functions}.
  The \emph{monodromy} of a germ $[f] \in \calS$ is a germ of functions $g \in \rmC^\infty(\Sigma; \C)$ defined by
  \begin{equation*}
    g(c) = f \circ \gamma(2 \pi) - f \circ \gamma(0)
  \end{equation*}
  for any lift $\wt \gamma \colon [0, 2\pi] \to \Sigma$ of a loop $\gamma$ in $\Sigma$ with $[\gamma] \in \pi_1(\Sigma, c)$ the counter-clockwise generator.
  A germ $[f] \in \calS$ is \emph{single-valued} iff its monodromy vanishes, or equivalently $f = \tilde{f} \circ \pi$ for some $\tilde{f} \in \rmC^\infty(\Sigma; \C)$.
  In this case, say $[f] \in \calS$ is \emph{smooth at the origin} if $\tilde{f}$ extends to a function in $\rmC^\infty(\C; \C)$.
\end{definition}

Therefore, $A_2$ in \cref{eq:action-local-form-basis} is better understood as an element in $\calS$ representing a multi-valued function $U_i \to \R$.

\begin{remark}
  For brevity, we do not distinguish $f$ with $\tilde{f}$ but instead speak about which spaces they live in and do not distinguish a germ with its representative.
\end{remark}

\begin{lemma} \label{lem:local-affine-equiv}
  Let $F$ be a semitoric system on $(M, \om)$ and $F'$ on $(M', \om')$ of dimension four and let $B = F(M)$ and $B' = F'(M')$.
  Then for $\tilde{c} \in B_\rmf$ and $\tilde{c}' \in B'_\rmf$, there is $U \in \calN(B, \tilde{c})$, $U' \in \calN(B', \tilde{c}')$, and a diffeomorphism $G \colon U \to U'$ sending $\tilde{c}$ to $\tilde{c}'$ such that $G^* \Res{\check L^{F'}}_{U'} = \Res{\check L^F}_U$ iff there are germs of smooth functions $G_j, G'_j \in \frakV_+$, $S'_0, S_0 \in \rmC^\infty(\C; \C)$ for $j \in \Z_m$ satisfying
  \begin{align*}
    \sfs_j &= \Tl_0[S_0 \circ G_j^{-1}] + (2\pi X)\Z, \quad \sfg_{j, \ell} = \Tl_0[\proj_2 \circ G_\ell \circ G_j^{-1}], \\
    \sfs'_j &= \Tl_0[S'_0 \circ G'_0 \circ (G'_j)^{-1}] + (2\pi X)\Z, \quad \sfg'_{j, \ell} = \Tl_0[\proj_2 \circ G'_\ell \circ (G'_j)^{-1}]
  \end{align*}
  for $j, \ell \in \Z_m$ such that
  \begin{align} \label{eq:affine-equiv}
    S'_0 \circ G'_0 - S_0 = \sum_{j \in \Z_m} \Im(G'_j \ln G'_j) - \sum_{j \in \Z_m} \Im(G_j \ln G_j) - \sum_{j \in \Z_m} \Im(G'_j - G_j) \in \calS.
  \end{align}
\end{lemma}

\begin{proof}
  By \cref{thm:lattice-bundle-period-action-stor} we have $G^* \Res{\check L^{F'}}_{U'} = \Res{\check L^F}_U$ iff $G^* \Res{\calA^{F'}}_{U'} = \Res{\calA^F}_U$.
  Let $(A_1, A_2)$ generate $\Res{\calA^F}_U$ and $(A'_1, A'_2)$ generate $\Res{\calA^{F'}}_{U'}$ in the sense of \cref{eq:action-local-form} with $A_1 = A'_1 = c^1$, and then $\der G^* A'_2 \in \der A_2 + \der c^1\Z$.
  By \cref{eq:action-local-form-basis}, with $G_j = E_j \circ E_0^{-1}$ and $G'_j = E'_j \circ G \circ E_0^{-1}$, we have
  \begin{align*}
    2\pi (E_0^{-1})^*A_2 &= -\sum_{j \in \Z_m} \Im(G_j \ln G_j - G_j) + \wt S_0,\\
    2\pi (E_0^{-1})^*G^* A'_2 &= -\sum_{j \in \Z_m} \Im(G'_j \ln G'_j - G'_j) + \wt S'_0 \circ G'_0
  \end{align*}
  where
  \begin{align*}
    \tsfs_j = \Tl_0[\wt S_0 \circ G_j^{-1}], \quad \tsfs'_j = \Tl_0[\wt S'_0 \circ G'_0 \circ (G'_j)^{-1}]
  \end{align*}
  for $j \in \Z_m$.
  The equality \cref{eq:affine-equiv} holds.
\end{proof}

\begin{definition} \label{def:affine-equiv-complete}
  Two tuples $(G_0, \dotsc, G_{m-1}, \wt S_0)$ and $(G'_0, \dotsc, G'_{m-1}, \wt S'_0)$ in $\frakV_+^m \times \rmC^\infty(\C; \C)$ are called \emph{affine equivalent} iff
  \begin{align*}
    -\sum_{j \in \Z_{m_i}} \Im(G_j \ln G_j - G_j) + \wt S_0 &= -\sum_{j \in \Z_{m_i}} \Im(G'_j \ln G'_j - G'_j) + \wt S'_0 \circ G'_0 \in \calS
  \end{align*}
  Two tuples $(G_0, \dotsc, G_{m-1})$ and $(G'_0, \dotsc, G'_{m-1})$ in $\frakV_+^m$ are called \emph{affine admissible} iff
  \begin{align} \label{eq:affine-adm-complete}
    \sum_{j \in \Z_{m_i}} \Im(G'_j \ln G'_j - G'_j) - \sum_{j \in \Z_{m_i}} \Im(G_j \ln G_j - G_j) \in \rmC^\infty(\C; \C).
  \end{align}
\end{definition}

\begin{lemma} \label{lem:affine-equiv-adm}
  Two tuples $(G_0, \dotsc, G_{m-1})$ and $(G'_0, \dotsc, G'_{m-1})$ in $\frakV_+^m$ are affine admissible iff for any $\wt S_0 \in \rmC^\infty(\C; \C)$ there is a unique $\wt S'_0 \in \rmC^\infty(\C; \C)$ making $(G_0, \dotsc, G_{m-1}, \wt S_0)$ and $(G'_0, \dotsc, G'_{m-1}, \wt S'_0)$ affine equivalent.
\end{lemma}

\begin{proof}
  If \cref{eq:affine-adm-complete} is some smooth function then it equals $\wt S'_0 \circ G'_0 - \wt S_0$.
\end{proof}

\begin{definition} \label{def:affine-equiv-jet}
  Two tuples $(\sfG_0, \dotsc, \sfG_{m-1}, \wt \sfS_0)$ and $(\sfG'_0, \dotsc, \sfG'_{m-1}, \wt \sfS'_0)$ in $\sfV_+^m \times \tsfR$ are called \emph{affine equivalent} iff there are $G_j, G'_j \in \frakV_+$, $\wt S'_0, \wt S_0 \in \rmC^\infty(\C; \C)$ with
  \begin{align*} \label{def:diff-to-jet}
    \wt \sfS_0 = \Tl_0[\wt S_0], \quad \sfG_j = \Tl_0[G_j], \quad \wt \sfS'_0 = \Tl_0[\wt S'_0], \quad \sfG'_j = \Tl_0[G'_j]
  \end{align*}
  for $j \in \Z_m$, such that $(G_0, \dotsc, G_{m-1}, \wt S_0)$ and $(G'_0, \dotsc, G'_{m-1}, \wt S'_0)$ are affine equivalent; then $(\sfG_0, \dotsc, \sfG_{m-1})$ and $(\sfG'_0, \dotsc, \sfG'_{m-1})$ in $\frakV_+^m$ are called \emph{affine admissible} iff $(G_0, \dotsc, G_{m-1})$ and $(G'_0, \dotsc, G'_{m-1})$ are affine admissible.
\end{definition}

\begin{lemma} \label{lem:affine-equiv-jet}
  If $(\sfG_0, \dotsc, \sfG_{m-1}, \wt \sfS_0)$ and $(\sfG'_0, \dotsc, \sfG'_{m-1}, \wt \sfS'_0)$ in $\sfV_+^m \times \tsfR$ are affine equivalent, then for any $G_j, G'_j \in \frakV_+$, $\wt S_0 \in \rmC^\infty(\C; \C)$ with $\wt \sfS_0 = \Tl_0[S_0]$, $\sfG_j = \Tl_0[G_j]$, and $\sfG'_j = \Tl_0[G'_j]$ for $j \in \Z_m$, there is a unique $\wt S'_0 \in \rmC^\infty(\C; \C)$ with $\wt \sfS'_0 = \Tl_0[S'_0]$ making $(G_0, \dotsc, G_{m-1}, \wt S_0)$ and $(G'_0, \dotsc, G'_{m-1}, \wt S'_0)$ affine equivalent.
\end{lemma}

\begin{proof}
  By \cite[Lemma 2.18]{PeTa2018}, both sides of \cref{eq:affine-adm-complete} change by adding $\inftes$ when $G_j$, $G'_j$ for $j \in \Z_m$, $\wt S_0$, and $\wt S'_0$ change by adding $\inftes$.
\end{proof}

\begin{definition} \label{def:affine-equiv-label}
  Two complete focus-focus labels $[\tsfs_j, \sfg_{j, \ell}]_{j, \ell \in \Z_m}$ and $[\tsfs'_j, \sfg'_{j, \ell}]_{j, \ell \in \Z_{m'}}$ are called \emph{affine equivalent} via $\sfG \in \sfV_+$ iff $m = m'$ and $(\sfG_0, \dotsc, \sfG_{m-1}, \wt \sfS_0)$ and $(\sfG'_0, \dotsc, \sfG'_{m-1}, \wt \sfS'_0)$ are affine equivalent where
  \begin{align*}
    \sfG_j(X, Y) = \Pa{X, \sfg_{0, j}(X, Y)}, \qquad \sfG'_j(X, Y) = \Pa{X, \sfg'_{0, j}(\sfG(X, Y))}.
  \end{align*}
\end{definition}

\section{Main results}
\label{sec:main-result}

\begin{theorem} \label{thm:global-affine-equiv}
  Let $F$ be a semitoric system on $(M, \om)$ and $F'$ on $(M', \om')$ of dimension four and let $B = F(M)$ and $B' = F'(M')$.
  Then a diffeomorphism $G \colon B \to B'$ satisfies $G^* \check L^{F'} = \check L^F$ iff there are complete semitoric ingredient representatives
  \begin{equation*}
    \Pa{\De, (c_i)_{i = 1}^\vf, \Pa{[\tsfs^i_j, \sfg^i_{j, \ell}]_{j, \ell \in \Z_{m_i}}}_{i = 1}^\vf}, \Pa{\De', (c'_i)_{i = 1}^\vf, \Pa{[\tsfs^{\prime i}_j, \sfg^{\prime i}_{j, \ell}]_{j, \ell \in \Z_{m'_i}}}_{i = 1}^\vf}
  \end{equation*}
  where $\De = \De'$, $(c_i)_{i = 1}^\vf = (c'_i)_{i = 1}^\vf$, and $[\tsfs^i_j, \sfg^i_{j, \ell}]_{j, \ell \in \Z_{m_i}}$ is affine equiavlent to\linebreak $[\tsfs^{\prime i}_j, \sfg^{\prime i}_{j, \ell}]_{j, \ell \in \Z_{m'_i}}$ via $\sfG = \Tl_0[G]$.
\end{theorem}

\begin{proof}
  This is almost a result of \cref{lem:local-affine-equiv,lem:affine-equiv-adm,lem:affine-equiv-jet}, except that since $G$ preserves the period lattice globally we should use $G^* A'_2 = A_2$ instead of $\der G^* A'_2 \in \der A_2 + \der c^1\Z$.
\end{proof}

\subsection{Fibers with one pinched point}
\label{ssec:main-theorem-single-pinch}

Throughout this section, we deal with focus-focus fibers with multiplicity $1$.

\begin{definition}
  A semitoric system $F = (J, H)$ is \emph{simple} if each fiber of $J$ (and hence also each fiber of $F$) contains at most one focus-focus point.
\end{definition}

\begin{lemma} \label{lem:affine-equiv-jet-single}
  Two tuples $(\sfG_0, \wt \sfS_0)$ and $(\sfG'_0, \wt \sfS'_0)$ in $\sfV_+ \times \tsfR$ are affine equivalent iff there are are equal.
\end{lemma}

\begin{proof}
  It is sufficient to prove the case $\sfG_0(Z) = Z$ and $\sfS_0(Z) = 0$.
  By \cref{def:affine-equiv-jet} there is $G'_0 \in \frakV^+$ with $\sfG'_0 = \Tl_0[G'_0]$ such that
  \begin{equation*}
    \Im(G'_0 \ln G'_0 - G'_0) - \Im(Z \ln Z - Z) = \wt S'_0 \circ G'_0 \in \calS
  \end{equation*}
  is smooth and differentiating we get
  \begin{equation*}
    (G'_0)^* \der \wt S'_0 = (G'_0)^* \kappa - \kappa
  \end{equation*}
  is smooth.
  Hence the translation along the fibers by $(G'_0)^* \kappa$
  \begin{equation*}
    \Psi_{(G'_0)^* \kappa} = \Psi_\kappa \circ \Psi_{(G'_0)^* \der \wt S'_0}
  \end{equation*}
  is smooth since both maps on the right-hand side are smooth, by~\cite[Lemma~2.17]{PeTa2018}.
  Then, by~\cite[Lemma~2.19]{PeTa2018}, we have $\sfG_0(Z) = Z$ and then $\sfS'_0(Z) = 0$.
\end{proof}

\cref{lem:affine-equiv-jet-single} is an easy corollary of~\cite[Lemma~2.17, Lemma~2.19]{PeTa2018}, but these two lemmas we use here are not new in~\cite{PeTa2018}.
The earliest proof of essentially the same lemmas, as far as we know, is in~\cite{MR1941440}.

\begin{corollary} \label{lem:affine-equiv-label-single}
  Two complete focus-focus labels $[\wt \sfs_0, \sfg_{0, 0}]$ and $[\wt \sfs'_0, \sfg'_{0, 0}]$ in $\tsfR \times \sfR_+$ are affine equivalent via $\sfG \in \sfV_+$ iff they are equal and $\sfG(X, Y) = (X, Y)$.
\end{corollary}

\begin{proof}
  This is a corollary of \cref{lem:affine-equiv-jet-single} in view of \cref{def:affine-equiv-label}.
\end{proof}

\begin{theorem} \label{thm:global-affine-equiv-single}
  Two semitoric systems $F$ on $(M, \om)$ and $F'$ on $(M', \om')$ of dimension four with every focus-focus fiber once-pinched are affine equivalent via a diffeomorphism $G \colon B \to B'$ iff they are isomorphic via $G$.
\end{theorem}

\begin{proof}
  This is a result of \cref{lem:affine-equiv-label-single,thm:global-affine-equiv}.
\end{proof}

\begin{corollary} \label{cor:global-affine-equiv-simple}
  Two simple semitoric systems $F$ on $(M, \om)$ and $F'$ on $(M', \om')$ of dimension four are affine equivalent via a diffeomorphism $G \colon B \to B'$ iff they are isomorphic.
\end{corollary}

\begin{proof}
  This is a corollary of \cref{thm:global-affine-equiv-single}.
\end{proof}

\subsection{First-order invariants}
\label{ssec:first-order}

We realize that the first-order smooth invariants and liftable germs of diffeomorphisms in Bolsinov--Izosimov~\cite{MR4057723} play important roles in our affine classification, and we cite their definitions here.

Let $G \in \frakV_+$ with the linearization at the origin $D_0 G = G(x + \imag y) = (x + \imag (ax + by))$ for some $a \in \R$ and $b > 0$ and recall $q$ in \cref{def:local-model-focus}.
Then the linearization of $X_{\proj_2 \circ G \circ q}$ is
\begin{equation*}
  D_0 X_{\proj_2 \circ G \circ q} = \begin{pmatrix}
    -b & -a &  0 &  0 \\
     a & -b &  0 &  0 \\
     0 &  0 &  b & -a \\
     0 &  0 &  a &  b
  \end{pmatrix}
\end{equation*}
with four eigenvalues $b \pm \imag a, -b \pm \imag a$.

Let $\sfG = \Tl_0[G] \in \sfV_+$; we assign a complex formal series $\sfG_\C(Z, \cj{Z}) = X + \imag \sfG(X, Y)$ by setting $Z = X + \imag Y$ with expansion
\begin{equation*}
  \sfG_\C(Z, \cj{Z}) = \sum_{p, q = 0}^\infty \sfG_\C^{(p, \cj{q})} Z^p \cj{Z}\vphantom{Z}^q
\end{equation*}
where $\sfG_\C^{(p, \cj{q})} \in \C$, $\sfG_\C^{(0, \cj{0})} = 0$, and $\sfG_\C^{(0, \cj{1})} + \cj{\sfG_\C^{(1, \cj{0})}} = 1$.

\begin{definition}
  Analoguous to~\cite[Proposition 3.14]{MR4057723}, let
  \begin{equation*}
    \mu = \sfG_\C^{(0, \cj{1})} \Big/ \cj{\sfG_\C^{(1, \cj{0})}} = \frac{1 - b + \imag a}{1 + b - \imag a} \in \bbB^2
  \end{equation*}
  and call it the \emph{first-order invariant} of $G$.
  For any $\mu \in \bbB^2$ let
  \begin{equation*}
    Z_\mu = \frac{1}{1 + \cj{\mu}} Z + \frac{\mu}{1 + \mu} \cj{Z}, \qquad \cj{Z}_\mu = \frac{\cj{\mu}}{1 + \cj{\mu}} Z + \frac{1}{1 + \mu} \cj{Z}.
  \end{equation*}
\end{definition}

We observe that the unique linear element in $\frakV_+$ with first-order invariant $\mu \in \bbB^2$ is $G_\mu$ defined by $G_\mu(Z) = Z_\mu$.
There is another expansion under coordinates $(Z_\mu, \cj{Z}_\mu)$ as
\begin{equation*}
  G_\C(Z, \cj{Z}) = \sum_{p, q = 0}^\infty \sfG_{\C, \mu}^{(p, \cj{q})} Z_\mu^p \cj{Z}_\mu^q.
\end{equation*}

\begin{definition} \label{def:mu-liftable-jet}
  For any $\mu \in \bbB^2$ we say that $\sfG \in \sfV_+$ is \emph{$\mu$-liftable} iff $\sfG_{\C, \mu}^{(0, \cj{q})} = 0$ for any $q \geq 0$; in particular $G$ is called \emph{liftable} in the sense of~\cite[Theorem 3.3]{MR4057723} if it is $0$-liftable.
  We say that $\sfG \in \sfV_+$ is \emph{$\mu$-holomorphic} iff $\sfG_{\C, \mu}^{(p, \cj{q})} = 0$ for any $p \geq 0$ and $q \geq 1$.
\end{definition}

\begin{remark}
  Being $\mu$-holomorphic implies being $\mu$-liftable.
\end{remark}

\subsection{Fibers with multiple pinched points}
\label{ssec:main-theorem-multiple-pinch}

\begin{lemma} \label{lem:local-affine-lnG}
  For any $G \in \frakV_+$ with first-order invariant $\mu$, we have that $G \ln G \in \calS$ is a linear combination of $Z_\mu^{-k}$ for $k \in \N$, $1$, and $\ln Z_\mu$ with coefficients in $\rmC^\infty(\C; \C)$.
\end{lemma}

\begin{proof}
  We compute directly,
  \begin{align*}
    G \ln G &= \sum_{l = 1}^\infty \frac{(-1)^{l-1}}{l} \Pa{\sum_{p+q\geq2} \sfG_{\C, \mu}^{(p, \cj{q})} Z_\mu^p \cj{Z}_{\mu}^q}^{l+1} Z_\mu^{-l} + G \ln Z_\mu + \rmC^\infty(\C; \C).
  \end{align*}
\end{proof}

\begin{proposition} \label{prop:local-affine-equiv-mult-liftable}
  If $(\sfG_0, \dotsc, \sfG_{m-1}), (\sfG'_0, \dotsc, \sfG'_{m-1}) \in \sfV_+^m$, all with first-order invariant $\mu$, and all $\sfG_j$ are $\mu$-liftable then the two tuples are affine admissible iff all $\sfG'_j$ are $\mu$-liftable for $j \in \Z_m$ and
  \begin{equation} \label{eq:G-equal-sum}
    \sum_{j \in \Z_m} \sfG_j = \sum_{j \in \Z_m} \sfG'_j.
  \end{equation}
\end{proposition}

\begin{proof}
  Without loss of generality, we assume $\mu = 0$ and then $Z_\mu = Z$.
  Let
  \begin{align*}
    \sfH &= \sum_{j \in \Z_m} \sum_{l = 1}^\infty \frac{(-1)^{l-1}}{l} \Im \Pa{\Pa{\sum_{p+q\geq2} \sfG_{j\C}^{(p, \cj{q})} Z^{p-1} \cj{Z}^q}^{l+1} Z}, \\
    \sfH' &= \sum_{j \in \Z_m} \sum_{l = 1}^\infty \frac{(-1)^{l-1}}{l} \Im \Pa{\Pa{\sum_{p+q\geq2} \sfG_{j\C}^{\prime (p, \cj{q})} Z^{p-1} \cj{Z}^q}^{l+1} Z}.
  \end{align*}
  Then
  \begin{align*}
    \sum_{j \in \Z_m} \Im \Pa{G_j \ln G_j} &= \sfH + \sum_{j \in \Z_m} \Im \Pa{G_j \ln Z} + \rmC^\infty(\C; \C), \\
    \sum_{j \in \Z_m} \Im \Pa{G'_j \ln G'_j} &= \sfH' + \sum_{j \in \Z_m} \Im \Pa{G'_j \ln Z} + \rmC^\infty(\C; \C).
  \end{align*}
  Note that the terms with $\ln Z$ coincide iff \cref{eq:G-equal-sum}.
  We claim that $G_j$ (resp. $G'_j$) is liftable for every $j \in \Z_m$ iff $\sfH \in \tsfR$ (resp. $\sfH' \in \tsfR$), and the result would hold.

  Here is the proof of the claim.
  If $G_j$ is liftable for every $j \in \Z_m$ then
  \begin{align*}
    \sfH &= \sum_{j \in \Z_m} \sum_{l = 1}^\infty \frac{(-1)^{l-1}}{l} \Im \Pa{\Pa{\sum_{p\geq1, p+q\geq2} \sfG_{j\C}^{(p, \cj{q})} Z^{p-1} \cj{Z}^q}^{l+1} Z} \in \tsfR.
  \end{align*}
  On the other hand, for any $l \geq 1$ the terms
  \begin{align*}
    \sfH^{(-l, \cj{2l+2})} &= \sum_{j \in \Z_m} \frac{(-1)^{l-1}}{2 l \imag} \Pa{\sfG_{j\C}^{(0, \cj{2})}}^{l+1}, &\sfH^{(2l+2, \cj{-l})} &= \sum_{j \in \Z_m} \frac{(-1)^l}{2 l \imag} \Pa{\cj{\sfG}_{j\C}^{(0, \cj{2})}}^{l+1}.
  \end{align*}
  If $\sfH \in \tsfR$, then for $l \geq 1$ we have
  \begin{align*}
    \sum_{j \in \Z_m} \Pa{\sfG_{j\C}^{(0, \cj{2})}}^{l+1} = \sum_{j \in \Z_m} \Pa{\cj{\sfG}_{j\C}^{(0, \cj{2})}}^{l+1} = 0
  \end{align*}
  and by the Vandermont determint $\sfG_{j\C}^{(0, \cj{2})} = 0$ for every $j \in \Z_m$.
  Now we have
  \begin{align*}
    \sfH^{(-l, \cj{3l+3})} &= \sum_{j \in \Z_m} \frac{(-1)^{l-1}}{2 l \imag} \Pa{\sfG_{j\C}^{(0, \cj{3})}}^{l+1}, &\sfH^{(3l+3, \cj{-l})} &= \sum_{j \in \Z_m} \frac{(-1)^l}{2 l \imag} \Pa{\cj{\sfG}_{j\C}^{(0, \cj{3})}}^{l+1}
  \end{align*}
  for any $l \geq 1$ and then since $\sfH \in \tsfR$ we have
  \begin{align*}
    \sum_{j \in \Z_m} \Pa{\sfG_{j\C}^{(0, \cj{3})}}^{l+1} = \sum_{j \in \Z_m} \Pa{\cj{\sfG}_{j\C}^{(0, \cj{3})}}^{l+1} = 0
  \end{align*}
  and therefore $\sfG_{j\C}^{(0, \cj{3})} = 0$ for every $j \in \Z_m$.
  An induction process would imply that $\sfG_{j\C}^{(0, \cj{q})} = 0$ for any $q \in \N$ and therefore, $\sfG_j$ are liftable for every $j \in \Z_m$.
\end{proof}

\subsection{Counterexamples}
\label{ssec:counterexamples}

We point out that the affine structure \textbf{does not} determine the symplectic structure, by trivial examples.

\begin{example}
  Let $m \geq 3$.
  Let $\tsfl = [\tsfs_j, \sfg_{j, \ell}]_{j, \ell \in \Z_m}$ be a complete focus-focus label of multiplicity $m$ generated by distinct elements $\sfg_{j, j+1} \in \sfR_+$, $j \in \Z_m \setminus \Set{m-1}$, and an arbitrary element $\tsfs_0 \in \tsfR$.
  Let $\sigma \in S_m$ be a permutation that breaks the cyclic order and let $\tsfl' = [\tsfs'_j, \sfg'_{j, \ell}]_{j, \ell \in \Z_m}$ be a complete focus-focus label with $\sfg'_{j, \ell} = \sfg_{\sigma j, \sigma \ell}$ and $\tsfs'_j = \tsfs_{\sigma j}$ for $j, \ell \in \Z_m$.
  Then $\tsfl$ and $\tsfl'$ are different but affine equivalent via $(X, Y)$.
\end{example}

\begin{example}
  Let $m = 2$.
  Let $\sfG'_0, \sfG'_1$ be arbitrary liftable elements of $\sfV^+$, let $\sfG_1 = \sfG'_0 + \sfG'_1 - (X, Y)$ and let $\tsfs_0 \in \sfR_{2\pi X}$ be an arbitrary element.
  Let
  \begin{align*}
    \sfg_{0, 1} = \proj_1 \circ \sfG_1, \quad \sfg'_{0, 1} = \proj_1 \circ \sfG'_1 \circ (\sfG'_0)^{-1} \in \sfR_+
  \end{align*}
  and let
  \begin{align*}
    \tsfs'_0 = \tsfs_0 \circ (\sfG'_{0\C})^{-1}(Z) + \sum_{l = 1}^\infty \frac{(-1)^{l-1}}{l} \Im \Pa{\sfP_l \circ (\sfG'_{0\C})^{-1}(Z) \cdot (\sfG'_{0\C})^{-1}(Z)}
  \end{align*}
  where
  \begin{equation*} \begin{split}
    \sfP_l(Z) &= \Pa{\sum_{p\geq1, p+q\geq2} \sfG_{0\C}^{\prime (p, \cj{q})} Z^{p-1} \cj{Z}^q}^{l+1} \\
    &\phantom{{}=}+ \Pa{\sum_{p\geq1, p+q\geq2} \sfG_{1\C}^{\prime (p, \cj{q})} Z^{p-1} \cj{Z}^q}^{l+1}
      - \Pa{\sum_{p\geq1, p+q\geq2} \sfG_{1\C}^{(p, \cj{q})} Z^{p-1} \cj{Z}^q}^{l+1}.
  \end{split} \end{equation*}
  Let $\tsfl = [\tsfs_j, \sfg_{j, \ell}]_{j, \ell \in \Z_2}$ be a complete focus-focus label generated by $\sfg_{0, 1}$ and $\tsfs_0$ and let $\tsfl' = [\tsfs'_j, \sfg'_{j, \ell}]_{j, \ell \in \Z_2}$ be one generated by $\sfg'_{0, 1}$ and $\tsfs'_0$.
  Then $\tsfl$ and $\tsfl'$ are different but affine equivalent via $\sfG'_0$.
\end{example}

\begin{example}
  As a concrete example, let $m = 2$, let $a, b \in \C$ with $ab(a + b) \neq 0$, and
  \begin{align*}
    \sfG_{1\C}(Z) = Z + (a + b) Z \cj{Z}, \quad \sfG'_{0\C}(Z) = Z + a Z \cj{Z}, \quad \sfG'_{1\C}(Z) = Z + b Z \cj{Z}.
  \end{align*}
  Then by Lagrange inversion formula and direct calculation, we obtain
  \begin{align*}
    (\sfG'_{0\C})^{-1} (W) &= \sum_{l = 1}^\infty \frac{(-2a)^{l-1} (2l-3)!!}{l!} \Pa{\Re W - a (\Im W)^2}^l + \Im W \imag
  \end{align*}
  with setting $(-1)!! = 1$ and we let
  \begin{align*}
    \sfg_{0, 1}(Z) &= \Im Z, \qquad \sfg'_{0, 1}(W) = \Im W, \qquad \tsfs_0(Z) = 0, \\
    \tsfs'_0(W) &= \sum_{l = 1}^\infty \frac{(-1)^{l-1}}{l} \Im \Pa{\Pa{a^{l+1} + b^{l+1} - (a + b)^{l+1}} (\sfG'_{0\C})^{-1}(W) \cj{(\sfG'_{0\C})^{-1}(W)}^{l+1}}.
  \end{align*}
  Let $\tsfl = [\tsfs_j, \sfg_{j, \ell}]_{j, \ell \in \Z_m}$ be a focus-focus label of multiplicity $m$ generated by $\sfg_{0, 1}$ and $\tsfs_0$ and let $\tsfl' = [\tsfs'_j, \sfg'_{j, \ell}]_{j, \ell \in \Z_m}$ be a focus-focus label of multiplicity $m$ generated by $\sfg'_{0, 1}$ and $\tsfs'_0$.
  Then $\tsfl$ and $\tsfl'$ are different but affine equivalent via $\sfG'_0$.
\end{example}

The last theorem is an algorithm to generate nontrivial examples.

\begin{theorem} \label{thm:local-affine-equiv-label-mult}
  Let $\tsfl = [\tsfs_j, \sfg_{j, \ell}]_{j, \ell \in \Z_m}$ and $\tsfl' = [\tsfs'_j, \sfg'_{j, \ell}]_{j, \ell \in \Z_m}$ be complete focus-focus labels of multiplicity $m \in \N$.
  Let $\sfG_j(X, Y) = (X, \sfg_j(X, Y))$ and $\sfG'_j(X, Y) = (X, \sfg'_j(X, Y))$ for $j \in \Z_m$.
  If
  \begin{enumerate}
    \item both $\sfG_j$ and $\sfG'_j$ have the first-order invariant $\mu_j$ for $j \in \Z_m$,
    \item both $\sfG_j$ and $\sfG'_j$ are $\mu_j$-liftable for $j \in \Z_m$,
    \item both $\sfG_j$ and $\sfG'_j$ have the same sum for $j \in \Z_m$,
  \end{enumerate}
  then there is a unique $\tsfs''_j \in \tsfR$ for each $j \in \Z_m$ such that $\tsfl$ and $[\tsfs''_j, \sfg'_{j, \ell}]_{j, \ell \in \Z_m}$ are affine equivalent.
\end{theorem}

\begin{proof}
  This is a result of \cref{def:affine-equiv-label,lem:affine-equiv-adm,prop:local-affine-equiv-mult-liftable}.
\end{proof}

\end{document}